\documentclass[12pt]{amsart}
\usepackage[colorlinks=true,pagebackref,hyperindex,citecolor=green,linkcolor=red]{hyperref}
\usepackage{amsmath}
\usepackage{amsfonts}
\usepackage{amssymb}
\usepackage{color}
\usepackage[all]{xy}  
\usepackage{enumerate}
\usepackage[top=1in, bottom=1in, left=1in, right=1in]{geometry}
\usepackage{mathrsfs}

\usepackage{stmaryrd}

\theoremstyle{definition}
\newtheorem{theorem}{Theorem}[section]
\newtheorem{theoremx}{Theorem}
\newtheorem{question}[theorem]{Question}
\newtheorem{corollary}[theorem]{Corollary}
\newtheorem{lemma}[theorem]{Lemma}
\newtheorem{proposition}[theorem]{Proposition}

\theoremstyle{definition}
\newtheorem{definition}[theorem]{Definition}

\newtheorem{example}[theorem]{Example}

\newtheorem{conjecturex}[theoremx]{Conjecture}
\newtheorem{conjecture}[theorem]{Conjecture}
\newtheorem{remark}[theorem]{Remark}
\numberwithin{equation}{subsection}

\newcommand{\m}{\mathfrak{m}}
\newcommand{\n}{\mathfrak{n}}
\newcommand{\fpt}{{\operatorname{fpt}}}

\newcommand{\RR}{\mathbb{R}}
\newcommand{\NN}{\mathbb{N}}
\newcommand{\ZZ}{\mathbb{Z}}
\newcommand{\QQ}{\mathbb{Q}}

\newcommand{\cP}{\mathcal{P}}
\newcommand{\Cech}{ \check{\rm{C}}}

\newcommand{\reg}{\operatorname{reg}}
\newcommand{\pd}{\operatorname{pd}}
\newcommand{\sdim}{\operatorname{sdim}}

\newcommand{\Spec}{\operatorname{Spec}}
\newcommand{\Depth}{\operatorname{depth}}
\newcommand{\Hom}{\operatorname{Hom}}

\newcommand{\Supp}{\operatorname{Supp}}

\newcommand{\Max}{\operatorname{Max}}

\newcommand{\Ass}{\operatorname{Ass}}	
	
\newcommand{\depth}{\operatorname{depth}}

\newcommand{\Proj}{\operatorname{Proj}}

\newcommand{\FDer}[1]{\stackrel{#1}{\to}}

\newcommand{\ls}{\leqslant}%
\newcommand{\gs}{\geqslant}
\newcommand{\ds}{\displaystyle}

\newcommand{\p}{\mathfrak{p}}

\newcommand{\ov}[1]{\overline{#1}}
\newcommand{\ann}{\operatorname{ann}}

\newcommand{\cd}{\operatorname{cd}}
\newcommand{\lct}{\operatorname{lct}}
\renewcommand{\a}{\mathfrak{a}}
\newcommand{\OO}{\mathcal{O}}

\begin{document}
\newcommand{\tens}{\otimes}
\newcommand{\hhtest}[1]{\tau ( #1 )}
\renewcommand{\hom}[3]{\operatorname{Hom}_{#1} ( #2, #3 )}

\title{$F$-thresholds of graded rings}
\author{Alessandro De Stefani}
\author{Luis N\'u\~nez-Betancourt}


\maketitle

\begin{abstract}
The $a$-invariant, the $F$-pure threshold, and the diagonal $F$-threshold are three important invariants of a graded $K$-algebra.
Hirose, Watanabe, and Yoshida have conjectured relations among these invariants for strongly $F$-regular rings. 
In this article, we prove that these relations hold only assuming that  the algebra is $F$-pure. In addition, we present an interpretation of the $a$-invariant for $F$-pure Gorenstein graded $K$-algebras in terms of regular sequences that preserve $F$-purity. This result is in the spirit of Bertini theorems for projective varieties. Moreover, we show connections with projective dimension, Castelnuovo-Mumford regularity, and Serre's condition $S_k$. We also present analogous results and questions in characteristic zero. 
\end{abstract}

%%%%%%%%%%%%%%%%%%%%%%%%%%%%%%%%%%%%%%%%%%%%%%%%%%%%%%%%%
\section{Introduction}
%%%%%%%%%%%%%%%%%%%%%%%%%%%%%%%%%%%%%%%%%%%%%%%%%%%%%%%%%
Throughout this manuscript we consider an $F$-pure standard graded algebra over a field $K$ of positive characteristic $p$ such that $[K:K^p]<\infty$. 
%In this article, we focus on $F$-pure rings. 
We say that  $R$ is $F$-pure if the Frobenius map $F:R\to R$ splits. This property simplifies computations for cohomology groups and implies vanishing properties of these groups \cite{LyuVan,AnuragUliPure,MaRFmod}. 

%Let $\m$ denote the maximal homogeneous ideal of $R$. Suppose that $R$ is a reduced ring. We denote by $R^{1/p^e}$ the ring of $p^e$-roots, $\{r^{1/p^e} \mid r\in R\}$. We consider the mild assumption of $R$ being an $F$-finite ring, which means that $R^{1/p}$ is a finitely generated $R$-module.

In this article, we show relations among three important classes of invariants that give information about the singularity of $R$:   the $a$-invariants,  the $F$-pure threshold, and the diagonal $F$-threshold. 

We first consider the $i$-th $a$-invariant of $R,$ $a_i(R)$, 
which is defined as the degree of the highest nonzero part of the $i$-th local cohomology with support over $\m$ (see Section \ref{Background} for details). 
If $d=\dim(R)$, then $a_d(R)$ is often just called the $a$-invariant of $R$, and it is a classical invariant,  introduced by Goto and Watanabe \cite{GW1}. For example, if $R$ is Cohen-Macaulay, $a_d(R)$ determines the highest shift in the resolution of $R$, and moreover the Hilbert function and the Hilbert polynomial of $R$ coincide if and only if $a_d(R)$ is negative.
%It can be computed from the graded Betti numbers of $R$ seen as a graded module over a polynomial ring.

We also consider the $F$-pure threshold with respect to $\m$, $\fpt(R)$, which was introduced by Takagi and Watanabe \cite{TW2004} (see Section \ref{PreFPT} for details). This invariant is related to the log-canonical threshold and roughly speaking measures the asymptotic splitting order of $\m.$ 

Finally, we consider the diagonal $F$-threshold, $c^\m(R)$. This invariant is, roughly speaking, the asymptotic Frobenius order of $\m$, and it has several connections with tight closure theory \cite{HMTW}. The $F$-threshold has also connections with the Hilbert-Samuel multiplicity \cite{HMTW} and with the Hilbert Kunz multiplicity \cite{NBS}.
 
Hirose, Watanabe, and Yoshida \cite{HWY} made the following conjecture that relates these invariants:

\begin{conjecturex}[\cite{HWY}] \label{conj}
Let $R$ be a standard graded $K$-algebra with $K$ an $F$-finite field and $d=\dim(R).$
Assume that $R$ is strongly F-regular. Then,
\begin{enumerate}
\item $\fpt(R)\ls -a_d(R)\ls c^\m (R)$.
\item $\fpt(R)=-a_d(R)$ if and only if $R$ is Gorenstein.
\end{enumerate}
\end{conjecturex}

This conjecture has been proved only for strongly $F$-regular Hibi rings \cite{fptHibi} and affine toric rings \cite{HWY}. 
In addition, $\fpt(R)=-a_d(R)$  for strongly $F$-regular standard graded 
Gorenstein rings \cite[Example 2.4 (iv)]{TW2004}. In this paper we settle the first part of this conjecture and one direction of the second. Furthermore, we drop the restrictive hypotheses of strong $F$-regularity. We just need to assume that $R$ is $F$-pure, which is needed to define $\fpt(R)$ and $c^\m(R)$.  

\begin{theoremx}[see Theorems \ref{Thm a-inv}, \ref{Main c(R)}, and \ref{ThmGor}]\label{MainTheorem}
Let $R$ be a standard graded $K$-algebra which is $F$-finite and $F$-pure, and let $d=\dim(R)$. Then,
\begin{enumerate}
\item $\fpt(R)\ls -a_i(R)$ for every $i\in\NN$. 
\item If $a_i(R)\neq -\infty,$ then $-a_i(R)\ls c^\m(R)$. 
\item If $R$ is Gorenstein, then $\fpt(R)=-a_d(R).$
\end{enumerate}
\end{theoremx}

Theorem \ref{MainTheorem}(2) implies that $-a_d(R)\ls c^\m(R)$ for standard graded $F$-pure $K$-algebras. 
This inequality  has been proven before for complete intersections without assuming strongly $F$-regularity or $F$-purity \cite[Proposition 2.3]{LiSocles}. In Theorem \ref{Main c(R)}, we prove $-a_d(R)\ls c^\m_{-}(R)$ for all $F$-finite standard graded $K$-algebras.
In Example \ref{exnotGor}, we show that the converse of Theorem \ref{MainTheorem}(3) does not hold in general. This does not disprove Conjecture \ref{conj}(2), since the ring that we consider is not strongly 
$F$-regular.

We also prove the analogue of Conjecture \ref{conj}(1) in characteristic zero, and we show one direction of  Conjecture \ref{conj}(2). These results are obtained by reduction to positive characteristic methods \cite{HHCharZero}.

\begin{theoremx}[see Theorem \ref{MainCharZero}]
Let $K$ be a field of characteristic zero, and let $(R,\m,K)$ be a standard graded normal and $\QQ$-Gorenstein $K$-algebra such that $X=\Spec R$ is log-terminal. Let $d=\dim(R)$. Then,
\begin{enumerate}
\item $\lct(X)\ls -a_d(R)$. 
\item If $R$ is Gorenstein, then $\lct(X)=-a_d(R).$
\end{enumerate}
\end{theoremx}

Furthermore, we give a new interpretation of $\fpt(R)$ for Gorenstein standard graded $K$-algebras in terms of regular sequences that preserve $F$-purity. We call such a sequence an $F$-pure regular sequence (see Definition \ref{Def F-pure seq})

\begin{theoremx}[see Theorem \ref{ThmFpureSeq}]
Let $(R,\m,K)$ be a Gorenstein standard graded $K$-algebra which is $F$-finite and $F$-pure over an infinite field. Let $d=\dim(R)$, and let $s= \fpt(R)$. Then, there exists a regular sequence consisting of $s$ linear forms $\ell_1,\ldots, \ell_s$ such that $R/(\ell_1,\ldots,\ell_j)$
is $F$-pure for every $j=1,\ldots,s$.
\end{theoremx}

Theorem \ref{ThmFpureSeq} is in the spirit of Bertini type theorems and ladders on Fano varieties \cite{Ladder}. These theorems assert that `nice' singularities are still `good' after cutting by a general hyperplane. 
Among these nice singularities one encounters $F$-pure singularities of a projective variety \cite{BertiniFsing}.
Theorem \ref{ThmFpureSeq} gives a number of successive hyperplane cuts in $\Proj(R)$ that preserve $F$-purity. Furthermore, the hyperplane cuts remains globally $F$-pure.

Finally, we give explicit bounds for the projective dimension and the 
Castelnuovo-Mumford regularity of $R=S/I$, where $S=K[x_1,\ldots,x_n]$ is a polynomial ring over a field of positive characteristic, and $I \subseteq S$ is a homogeneous ideal such that $R$ is $F$-pure (see Theorem \ref{bounds}). 
In addition, we show that if an $F$-pure standard graded $K$-algebra $R=S/I$ satisfies Serre's $S_k$ condition, for some $k$ depending on the degrees of the generators of $I$, then $R$ is in fact Cohen-Macaulay (see Proposition \ref{PropSerre}).

%%%%%%%%%%%%%%%%%%%%%%%%%%%%%%%%%%%%%%%%%%%%%%%%%%%%%%%%%
\section{Background}\label{Background}
%%%%%%%%%%%%%%%%%%%%%%%%%%%%%%%%%%%%%%%%%%%%%%%%%%%%%%%%%

Throughout this article, $R$ denotes a commutative Noetherian ring with identity. A positively graded ring is a ring which admits a decomposition $R= \bigoplus_{i \gs 0} R_i$ of Abelian groups, with $R_{i} \cdot R_{j} \subseteq R_{i+j}$ for all $i$ and $j$. A {\it standard graded} ring is a positively graded ring such that $R_0=K$ is a field, $R = K[R_1]$ and $\dim_K (R_1) < \infty$, that is, $R$ is a finitely generated $K$-algebra, generated in degree one. We use the notation $(R,\m,K)$ to denote a standard graded $K$-algebra, where $\m = \bigoplus_{i \gs 1} R_i$ is the irrelevant maximal ideal. 

Suppose that $R$ is a standard graded $K$-algebra. A {\it graded module} is an $R$-module $M = \bigoplus_{n \in \ZZ} M_n$ such that $R_iM_j \subseteq M_{i+j}$. An $R$-homomorphism $\varphi:M \to N$ between graded $R$-modules is called {\it homogeneous of degree $c$} if $\varphi(M_i) \subseteq N_{i+c}$ for all $i \in \ZZ$. The set of all graded homomorphisms $M \to N$ of all degrees form a graded submodule of $\Hom_R(M,N)$. In general, these two modules are not the same, but they coincide when $M$ is finitely generated \cite{BrHe}.  Throughout this article, $E_R(K)$ will always denote the graded Matlis dual of $K$ in $R$. %that is the graded $R$-module $\Hom_K(R,K) = \bigoplus_i \Hom_K(R_i,K)$.
The grading  is given by $E_R(K)_i=\Hom_K(R_{-i},K)$ for all $i\in\ZZ.$

Let $I$ be a homogeneous ideal generated by the forms $f_1,\ldots,f_\ell\in R$. Consider the $\check{\mbox{C}}$ech complex, $\Cech^\bullet(\underline{f};R)$:
$$
0\to R\to \bigoplus_i R_{f_i}\to\bigoplus_{i,j} R_{f_i f_j}\to \ldots \to R_{f_1 \cdots f_\ell} \to 0,
$$
where $\Cech^i(\underline{f};R)=\bigoplus \limits_{1 \ls j_1<\ldots<j_i\ls \ell} R_{f_{j_1}\cdots f_{j_i}}$ and the homomorphism 
in every summand is a localization map with an appropriate sign.
Let $M$ be a graded $R$-module. 
We define the {\it $i$-th local cohomology of $M$ with support in $I$} by
$H^i_I(M):=H^i(\Cech^\bullet(\underline{f};R)\otimes_R M)$.
The local cohomology module $H^i_I(M)$ does not depend on the choice of generators, $f_1,\ldots,f_\ell$, of $I$.
Moreover, it only depends on the radical of $I$.
%Let $I \subseteq R$ be an ideal. Recall that the $i$-th local cohomology functor $H^i_I(-)$ is the $i$-th derived functor of $\Gamma_I(-)$, where $\Gamma_I(M) = \{m \in M \mid I^n m = 0$ for some $n \in \NN\}$. 
Since $M$ is a graded $R$-module and $I$ is homogeneous, the $i$-th local cohomology $H^i_I(M)$ is graded as well. 
Furthermore, if $\varphi:M \to N$ is a homogeneous $R$-module homomorphism of degree $d$, then the induced $R$-module map $H^i_I(M) \to H^i_I(N)$ is homogeneous of degree $d$ as well.

Assume that $R$ has positive characteristic $p$. For $e \in \NN$, let $F^e:R \to R$ denote the $e$-th iteration of the Frobenius endomorphism on $R$. If $R$ is reduced, $R^{1/p^e}$ denotes the ring of $p^e$-th roots of $R$ and we can identify $F^e$ with the inclusion $R \subseteq R^{1/p^e}$. When $R$ is standard graded, we view $R^{1/p^e}$ as a $\frac{1}{p^e}\NN$-graded module. In fact, if $r^{1/p^e} \in R^{1/p^e}$, then we  write $r \in R$ as $r=r_{d_1}+\ldots+r_{d_n}$, with $r_{d_j} \in R_{d_j}$. Then $r^{1/p^e} = r_{d_1}^{1/p^e}+\ldots+r_{d_n}^{1/p^e}$, and each $r_{d_j}^{1/p^e}$ has degree $d_j/p^e$. Similarly, if $M$ is a $\ZZ$-graded $R$-module, we have that $M^{1/p^e}$ is a $\frac{1}{p^e}\ZZ$-graded $R$-module. Here $M^{1/p^e}$ denotes the $R$-module which has the same additive structure as $M$, and multiplication defined by $r \cdot m^{1/p^e} := (r^{p^e}m)^{1/p^e}$ for all $r \in R$ and $m^{1/p^e} \in M^{1/p^e}$.

\begin{remark} \label{frac_grading} When $R$ is reduced, for any integer $e \gs 1$ we have an inclusion $R \hookrightarrow R^{1/p^e}$. As a submodule of $R^{1/p^e}$, $R$  inherits a natural $\frac{1}{p^e} \NN$ grading, which is compatible with the standard grading.
\end{remark}

If $\sqrt{I} = \m$ and $M$ is finitely generated, the modules $H^i_\m(M)$ are Artinian. Therefore, the following numbers are well defined.

\begin{definition} 
Let $M$ be a $\frac{1}{p^e}\NN$-graded finitely generated $R$-module.
For $i \in \NN$, if $H^i_\m(M) \ne 0$ then the $a_i$-invariant of $M$ is defined as
\[
\ds a_i(M) := \max\left\{\alpha \in \frac{1}{p^e}\ZZ \mid H^i_\m(M)_\alpha \ne 0\right\}.
\]
If $H^i_\m(M) = 0$, we set $a_i(M) := -\infty$.
\end{definition}

\begin{remark} 
With the grading introduced above, for a finitely generated graded $R$-module, $M$, we have that $a_i(M^{1/p^e}) = a_i(M)/p^e$ for all $i \in \NN$. In fact, $H^i_\m(M^{1/p^e}) \cong H^i_\m(M)^{1/p^e}$ since the functor $(-)^{1/p^e}$ is exact.
\end{remark}

\begin{definition} \label{DefnFSing} 
Let $R$ be a Noetherian ring of positive characteristic $p$. We say that $R$ is  {\it $F$-finite} if it is a finitely generated $R$-module via the action induced by the Frobenius endomorphism $F: R \to R$.  When $R$ is reduced, this is equivalent to say that $R^{1/p}$ is a finitely generated $R$-module.  If $(R,\m,K)$ is a standard graded $K$-algebra, then $R$ is $F$-finite if and only if $K$ is $F$-finite, that is, if and only if $[K:K^p]< \infty$. $R$ is called {\it $F$-pure} if $F$ is a pure homomorphism, that is $F \otimes 1: R \otimes_R M \to R \otimes_R M$ is injective for all $R$-modules $M$. $R$ is  called {\it $F$-split} if $F$ is a split monomorphism. 
\end{definition}
\begin{remark} \label{FpureNeg} 
Let $(R,\m,K)$ be a standard graded $F$-pure $K$-algebra. Then necessarily $a_i(R) \ls 0$ for all $i \in \ZZ$ \cite[Proposition $2.4$]{HRFpurity}. 

\end{remark}
\begin{remark}\label{FpureFsplit}
If $R$ is an $F$-pure ring, $F$ itself is injective and $R$ must be a reduced ring. We have that $R$ is $F$-split if and only if $R$ is a direct summand of $R^{1/p}$. If $R$ is an $F$-finite ring, $R$ is $F$-pure if and only $R$ is $F$-split (see \cite[Corollary $5.3$]{HRFpurity}). Since, throughout this article, we assume that $R$ is $F$-finite, we use the word $F$-pure to refer to both.
\end{remark}

\begin{definition}
An $F$-finite reduced ring  $R$ is strongly $F$-regular if for any element $f\in R \setminus \{0\}$, there exists $e\in \NN$ such that the inclusion $f^{1/p^e}R\to R^{1/p^e}$ splits.
\end{definition}

\begin{definition}
Let $S=K[x_1,\ldots,x_n]$ be a polynomial ring over an $F$-finite field.
The $e$-trace map, $\Phi_e:S^{1/p^e}\to S,$ is defined by
$$
\Phi_e\left(x^{\alpha_1/p^e}_1\cdots x^{\alpha_n/p^e}_n\right)=
\begin{cases} 
      x^{\frac{\alpha_1-p^e+1}{p^e}}_1\cdots x^{\frac{\alpha_n-p^e+1}{p^e}}_n & \alpha_1\equiv \ldots\equiv \alpha_n\equiv p^e-1 \hbox{ mod }p^e\\
      0 & \hbox{otherwise}
\end{cases}
$$
\end{definition}

We note that  $\Phi_{e'}\circ \Phi^{1/p^{e'}}_e=\Phi_{e'+e}$. Furthermore, $\Phi_e$ generates $\Hom_S(S^{1/p^e},S)$ as an $S^{1/p^e}$-module.

\begin{definition}[\cite{KarlCentersFpurity}]
Suppose that $R$ is an $F$-pure ring.
Let $\phi:R^{1/p^e}\to R$ be an $R$-homomorphism and let $J \subseteq R$ be an ideal. We say that $J$ is  {\it $\phi$-compatible} if $\phi(J^{1/p^e}) \subseteq J$. An ideal $J$ is said to be compatible if it is $\phi$-compatible for all $R$-linear maps $\phi:R^{1/p^e} \to R$ and all $e \in \NN$.
\end{definition}

We end this section by recalling an explicit description of $\Hom_R(R^{1/p^e},R)$ discovered by Fedder \cite{Fedder} that we use in the following sections.

\begin{remark}[{\cite[Corollary 1.5]{Fedder}}]\label{CorrespFedder}
Let $S$ be a polynomial ring over an $F$-finite field, and let  $\Phi_e: S^{1/p^e}\to S$ be the trace map. Let $I \subseteq S$ be a homogeneous ideal, and let $R=S/I$. We have a graded isomorphim
$$
\frac{(IS^{1/p^e}:_{S^{1/p^e}}I^{1/p^e})}{IS^{1/p^e}}\cong \Hom_R(R^{1/p^e},R)
$$
given by the correspondence $f^{1/p^e} \mapsto \varphi_{f,e}$, where
$\varphi_{f,e}: R^{1/p^e}\to R$ is defined by $\varphi_{f,e}(\overline{r}^{1/p^e})=\overline{\Phi_e(f^{1/p^e} r)}.$
\end{remark}

Thanks to the correspondence in Remark \ref{CorrespFedder}, we can make the following observation
\begin{remark} \label{compatibleFedder}
Let $S$ be a polynomial ring over an $F$-finite field. Let $I\subseteq S$ be a homogeneous ideal, and let $R=S/I$. Let $J\subseteq R$ be a homogeneous ideal, and let $\widetilde{J}$ denote its pullback in $S$.
We have that $J$ is compatible if and only if
$(I^{[p^e]}:I)\subseteq (\widetilde J^{[p^e]}:\widetilde J)$ for all $e\gs 1.$
\end{remark}

%%%%%%%%%%%%%%%%%%%%%%%%%%%%%%%%%%%%%%%%%%%%%%%%%%%%%%%%%
\section{Properties of $F$-thresholds}\label{PreFPT}
%%%%%%%%%%%%%%%%%%%%%%%%%%%%%%%%%%%%%%%%%%%%%%%%%%%%%%%%%

In this section we introduce basic definitions and properties for the diagonal $F$-threshold and the $F$-pure threshold.

\begin{definition}[\cite{HMTW}]
Suppose that $(R,\m,K)$ is  an  $F$-finite standard graded $K$-algebra. If $\nu_I(p^e)=\max\{r \in \NN \mid I^r\not\subseteq \m^{[p^e]}\}$, we define
$$c^\m_{-}(I)=\liminf\limits_{e\to \infty}\frac{\nu_I(p^e)}{p^e},\hbox{ and }c^\m_{+}(I)=\limsup\limits_{e\to \infty}\frac{\nu_I(p^e)}{p^e}.$$
If the two limits coincide, we denote the common value by $c^\m(I)$ and call it the $F$-threshold of $I$ with respect to $\m$. When $I=\m$, we call $c^\m(\m)$ the diagonal $F$-threshold of $R$, and we denote it by $c^\m(R)$. Similarly, we denote $c^\m_{-}(R) = c^\m_{-}(\m)$ and $c^\m_+(R) = c^\m_+(\m)$.
\end{definition}

We point out that the previous limit exits for $F$-pure rings \cite[Lemma $2.3$]{HMTW}.

\begin{definition}[\cite{TW2004}]
Let $(R,\m,K)$ be either a standard graded $K$-algebra or a local ring which is $F$-finite and $F$-pure, and let $I \subseteq R$ be an ideal (homogeneous in the former case). For a real number $\lambda\gs 0$, we say that $(R,I^\lambda)$ is $F$-pure if for every $e\gg 0$, there exists an element $f\in I^{\lfloor (p^e-1)\lambda\rfloor}$ such that the inclusion of $R$-modules $f^{1/p^e}R\subseteq R^{1/p^e}$ splits.
\end{definition} 
\begin{remark} Note that $(R,I^0) = (R,R)$ being $F$-pure simply means that $R$ is $F$-pure, according to Definition \ref{DefnFSing}.
\end{remark}

\begin{definition}[\cite{TW2004}]
Let $(R,\m,K)$ be either a standard graded $K$-algebra or a local ring which is $F$-finite and $F$-pure, and let $I \subseteq R$ be an ideal (homogeneous in the former case). The $F$-pure threshold of $I$ is defined by
$$
\fpt(I)=\sup\{\lambda\in \RR_{\gs 0} \mid  (R,I^{\lambda})\hbox{ is }F\hbox{-pure}\}.
$$
When $I=\m$, we denote the $F$-pure threshold by $\fpt(R)$.
\end{definition} 

\begin{definition}[\cite{AE}]
Let $(R,\m,K)$ be either a standard graded $K$-algebra or a local ring which is $F$-finite and $F$-pure. 
We define
$$
I_e(R):=\{r\in R \mid \varphi(r^{1/\p^e})\in\m\hbox{ for every }\varphi\in \Hom(R^{1/p^e},R)\}.
$$
In addition, we define the {\it splitting prime} of $R$ as $\cP(R) := \bigcap_e I_e(R)$ and the {\it splitting dimension} of $R$ to be $\sdim(R):=\dim(R/\cP(R))$.
\end{definition}

\begin{remark}\label{RemIeSplit}
We note that for a homogeneous element $r$,
 $r\not\in I_e(R)$ if and only if there is a map $\varphi\in \Hom_R(R^{1/p^e},R)$ such that $\varphi(r^{1/p^e})=1$.
 % because $\m$ is a homogeneous ideal, and we can pick $\varphi$ homogeneous.
\end{remark}

The following proposition gives basic  properties of the splitting prime for graded algebras. We include  details of the proof in the graded case for sake of completeness.
\begin{proposition}[\cite{AE}] \label{splitting prime properties}
Let $(R,\m,K)$ be an $F$-finite $F$-pure standard graded $K$-algebra.  Then
\begin{enumerate}
\item \label{basic1} $I_e(R)$ and $\cP(R)$ are homogeneous ideals.
\item \label{basic2} $\cP(R)$ is a prime ideal.
\item \label{basic3} $\cP(R)$ is the largest compatible homogeneous ideal of $R$, i.e. the largest homogeneous ideal of $R$ such that $\varphi(\cP(R)^{1/p^e}) \subseteq \cP(R)$ for all $\varphi \in \Hom_R(R^{1/p^e},R)$.
\item \label{basic4} $R/\cP(R)$ is strongly $F$-regular.
\item \label{basic5} $\cP(R)_\m = \cP(R_\m)$.
\end{enumerate}
\end{proposition}
\begin{proof} (\ref{basic1}) Let $e \gs 1$. Since both $R^{1/p^e}$ and $R$ are graded, and $R^{1/p^e}$ is a finitely generated $R$-module, we have that every homomorphism $R^{1/p^e} \to R$ is a sum of graded homomorphisms. Therefore, in the definition of $I_e(R)$ above, we can consider only graded homomorphisms. Let $r = r_0+r_1+\ldots+r_n \in I_e(R)$, with $r_i$ of degree $d_i$. Let $\varphi \in \Hom_R(R^{1/p^e},R)$ be homogeneous of degree $k$. Then
\[
\ds \varphi(r^{1/p^e}) = \varphi(r_0^{1/p^e})+\ldots+\varphi(r_n^{1/p^e}) \in \m,
\]
and each $\varphi(r_i^{1/p^e})$ now has degree $d_i+k$. Since $\m$ is 
homogeneous, we get $\varphi(r_i^{1/p^e}) \in \m$ for all $i=1,\ldots,n$, 
showing that $r_i \in I_e(R)$. Then $I_e(R)$ is a homogeneous ideal. 
We have that $\cP(R)$ is homogeneous is clear from its definition. The proofs of (\ref{basic2}), (\ref{basic3}) and (\ref{basic4}) are completely analogous to the ones in \cite[Theorem 3.3, and Theorem 4.7]{AE} and \cite[Remark 4.4]{KarlCentersFpurity} for local rings. For (\ref{basic5}), 
we note that $\varphi(\cP(R)_\m ^{1/p^e})\subseteq\cP(R)_\m$ for every 
$\varphi \in\Hom(R^{1/p^e}_\m,R_\m)$ because $R$ is $F$-finite.
Since $R_\m/(\cP(R)_\m)$ is strongly $F$-regular by (\ref{basic4}), we have that $\cP(R)_\m=\cP(R_\m).$
\end{proof}

\begin{definition}
Let $(R,\m,K)$ be an $F$-finite $F$-pure standard graded $K$-algebra. 
Let $J\subseteq R$ be a homogeneous ideal. Then, we define 
$$
b_J(p^e)=\max\{r \mid J^r\not\subseteq I_e(R)\}.
$$
\end{definition}

\begin{lemma}\label{Lemma Increase be}
Let $(R,\m,K)$ be a standard graded $K$-algebra which is $F$-finite and $F$-pure. Let $J\subseteq R$ be a homogeneous ideal. Then, $p\cdot b_J(p^e) \ls b_J(p^{e+1})$. 
\end{lemma}
\begin{proof}
Let  $f\in J^{b_J (p^e)}\smallsetminus I_e(R)$ be a homogeneous element.
Then, $Rf^{1/p^e}\to R^{1/p^e}$ splits as map of $R$-modules. 
Since $R$ is $F$-pure, there is a splitting $\alpha:R^{1/p^{e+1}}\to R^{1/p^e}$ as $R^{1/p^e}$-modules.
Then, 
$$
Rf^{1/p^e}\to R^{1/p^{e+1}}\FDer{\alpha} R^{1/p^e}
$$
splits as morphism of $R$-modules. Therefore, $Rf^{p/p^{e+1}}\to R^{1/p^{e+1}}$
splits as a map of $R$-modules. Hence, $f^p\in J^{p\cdot b_J(p^e)}\smallsetminus I_{e+1}(R)$, and so, $p\cdot b_J(p^e)\ls b_J(p^{e+1})$.
\end{proof}

We now present a characterization of the $F$-pure threshold that may be known to experts (see \cite[Key Lemma]{DanielMRL} for principal ideals). 
However, we were not able to find it in the literature  in the generality we need. We present the proof for the sake of completeness. This characterization is a key part for the proof of Theorem \ref{Thm a-inv}.

\begin{proposition}\label{PropEquivFPT}
Let $(R,\m,K)$ be a standard graded $K$-algebra which is $F$-finite and $F$-pure. Let $J\subseteq R$ be a homogeneous ideal. Then
$$
\fpt(J)=\lim\limits_{e\to \infty} \frac{b_J(p^e)}{p^e}.
$$
\end{proposition}
\begin{proof}
By the definition of $b_J(p^e),$ there exists $f\in J^{b_J(p^e)}\setminus I_e(R).$ Then, the map $R\to R^{1/p^e}$, defined by $1\mapsto f^{1/p^e}$ splits by Remark \ref{RemIeSplit}. Thus, $\frac{b_J(p^e)}{p^e}\in \{\lambda\in \RR_{\gs 0} \mid  (R,J^{\lambda})\hbox{ is }F\hbox{-pure}\}$. Hence, for all $e$, we have $\frac{b_J(p^e)}{p^e}\ls \fpt(J)$, and therefore $\left\{\frac{b_J(p^e)}{p^e}\right\}$ is a bounded sequence. By Lemma \ref{Lemma Increase be} we conclude that $\lim\limits_{e\to \infty} \frac{b_J(p^e)}{p^e}$ exists, and that $\lim\limits_{e\to \infty} \frac{b_J(p^e)}{p^e} \ls \fpt(J)$. 

Conversely, let $\sigma\in \{\lambda\in \RR_{\gs 0} \mid (R,J^{\lambda})\hbox{ is }F\hbox{-pure}\}$. For $e\gg 0$, we have that 
$J^{\lfloor (p^e-1)\sigma\rfloor}\not\subseteq I_e(R)$.
Then, $\frac{\lfloor (p^e-1)\sigma\rfloor}{p^e}\ls \frac{b_J(p^e)}{p^e}$ and thus
$$
\sigma=\sup\left\{\frac{\lfloor (p^e-1)\sigma\rfloor}{p^e}\right\}\ls \lim\limits_{e\to \infty} \frac{b_J(p^e)}{p^e}.
$$
Hence,
$$
\fpt(J)\ls \lim\limits_{e\to \infty} \frac{b_J(p^e)}{p^e}.
$$
\end{proof}

\begin{remark}
We note that analogous restatements of Proposition \ref{PropEquivFPT}  for $F$-finite $F$-pure local rings is also true, and the proof is essentially the same.
\end{remark}

%%%%%%%%%%%%%%%%%%%%%%%%%%%%%%%%%%%%%%%%%%%%%%%%%%%%%%%%%
\section{F-thresholds and $a$-invariants}
%%%%%%%%%%%%%%%%%%%%%%%%%%%%%%%%%%%%%%%%%%%%%%%%%%%%%%%%%
In this section, we prove the first part of our main theorem in positive characteristic. We start with a few preparation lemmas.

\begin{lemma}\label{LemmaSP}
Let $(S,\n,K)$ be a standard graded $F$-finite regular ring. Let $I\subseteq S$ be an ideal such that  $R=S/I$ is an $F$-pure ring. 
Let $\m=\n R$. Then, $\cP(R)=\m$ if and only if $(I^{[p^e]}:I)\subseteq (\n^{[p^e]}:\n)$ for all $e \gs 0$.
\end{lemma}
\begin{proof}
Let $\Phi_e:S^{1/p^e}\to S$ denote the $e$-trace map. If $(I^{[p^e]}:I)\subseteq (\n^{[p^e]}:\n)$,  then for every $f\in (I^{[p^e]}:I) \smallsetminus \n^{[p^e]}$ there exist a unit $u\in S$ and an element $g\in \n^{[p^e]}$ such that $f=ux^{p^e-1}_1\cdots x^{p^e-1}_n+g.$ Then, $f\cdot \n\subseteq \n^{[p^e]}$, and we get $\Phi_e(f^{1/p^e} \n^{1/p^e})\subseteq \n$.
Then by Remark \ref{CorrespFedder}, we have that $\varphi(\m^{1/p^e})\subseteq \m$ for every $\varphi:R^{1/p^e}\to R.$
Hence, $I_e(R)=\m$ for all $e \gs 1$ and $\cP(R) = \m$ as well. 

Conversely, if $\m =\cP(R)$, then $I_e(R)=\m$ for all $e \gs 1$.  This means that for every $e \gs 1$ and every $f\in (I^{[p^e]}:I)$, we have $\Phi_e(f^{1/p^e}\n^{1/p^e})\subseteq \n$ by Remark \ref{CorrespFedder}. Thus,  $f\cdot \n\subseteq \n^{[p^e]},$ and hence $f\in (\n^{[p^e]}:\n)$. 
\end{proof}

\begin{lemma}\label{Lemma Lift fpt}
Let $S=K[x_1,\ldots, x_n]$ be a polynomial ring over an $F$-finite field $K$. Let $\n=(x_1,\ldots,x_n)$ denote the maximal homogeneous ideal.
Let $I\subseteq S$ be a homogeneous ideal such that $R:=S/I$ is an $F$-pure ring, and let $\m = \n R$. Then,
$$
\min\left\{s\in\NN\ \bigg| \left[\frac{(I^{[p^e]}:I)+\n^{[p^e]}}{\n^{[p^e]}}\right]_s\neq 0\right\}=n(p^e-1)- b_\m(p^e)
$$
\end{lemma}
\begin{proof}
Let $\Phi_e:S^{1/p^e}\to S$ denote the $e$-trace map. Let $u:=\min\left\{s \in \NN \ \bigg| \left[\frac{(I^{[p^e]}:I)+\n^{[p^e]}}{\n^{[p^e]}}\right]_s\neq 0\right\}$ and  
$b= b_\m(p^e)$.
By our definition of $u$, there exists $f\in (I^{[p^e]}:I)\smallsetminus \n^{[p^e]}$, which is a homogeneous polynomial of degree $u$. Since $f\not\in \n^{[p^e]},$ there exists $x^\alpha\in\Supp \{f\}$ such that $x^\alpha\not \in \n^{[p^e]}$.
We pick $\beta=\boldsymbol{{\rm p^e-1}}-\alpha$, so $x^{\alpha}x^{\beta}=x^{\boldsymbol{{\rm p^e-1}}}$, where $\boldsymbol{{\rm p^e-1}}$ denotes the multi-index $(p^e-1,p^e-1,\ldots,p^e-1)$. We have that the map $\varphi:R^{1/p^e}\to R$ defined by 
$\varphi(\overline{r}^{1/p^e})=\overline{\Phi_e(f^{1/p^e}r^{1/p^e})}$ 
is a splitting of the Frobenius map on $R$ such that $\varphi((x^{\beta})^{1/p^e})=1$. Hence, $x^{\beta}\notin I_e(R)$. Since $|\beta|=(p^e-1)n-u,$ we have that $(p^e-1)n-u\ls b$; that is, $u \gs n(p^e-1)- b$.

For the other inequality, we pick a monomial $g\in \n$ of degree $b$ such that $\overline{g}\in\m^{b} \smallsetminus I_e(R)$.
Then, there exists a map $\varphi:R^{1/p^e}\to R$ such that $\varphi(\overline{g}^{1/p^e})=1.$ Therefore, there exists an element $f\in (I^{[q]}:I)\smallsetminus \n^{[p^e]}$ such that $\varphi(\overline{r}^{1/p^e})=\overline{\Phi_e(f^{1/p^e}r^{1/p^e} )}$ for all $\ov{r}^{1/p^e} \in R^{1/p^e}$. By definition of $\Phi_e,$ we have that $x^{\boldsymbol{{\rm p^e-1}}}\in\Supp(fg)$. Let $h$ be the homogeneous part of degree $(p-1)n-b$ of $f$. We note that $h\in  (I^{[p^e]}:I)$ because $I$ is homogeneous. In addition, $h\notin \n^{[p^e]}$ because $x^{\boldsymbol{{\rm p^e-1}}}\in\Supp(hg)$. Then we get $u\ls (p^e-1)n-b$, as desired.
\end{proof}

We re now ready to prove the first part of Theorem \ref{MainTheorem}.

\begin{theorem}\label{Thm a-inv}
Let $(R,\m,K)$ be a standard graded $K$-algebra which is $F$-finite and $F$-pure. Then $\fpt(R)\ls -a_i(R)$ for every $i\in\NN$.
\end{theorem}
\begin{proof} 
If $H^i_\m(R) = 0$ there is nothing to prove, since $a_i(R) = - \infty$. Let $i \in \NN$ be such that $H^i_\m(R) \ne 0$. Let $f \in\m^{b_\m(p^e)}\setminus I_e(R)$ be a homogeneous element, and let $\gamma=b_\m(p^e)/p^e$.
By Remark \ref{frac_grading} we can view $R$ as a $\frac{1}{p^e}\NN$-graded module. Then 
\[
\xymatrixcolsep{5mm}
\xymatrixrowsep{2mm}
\xymatrix{
R(-\gamma) \ \ar@{^{(}->}[rr]^-{f^{1/p^e}\cdot} && \ R^{1/p^e}
}
\]
splits, and the inclusion is homogeneous of degree zero. Applying the $i$-th local cohomology, we get a homogeneous split inclusion 
$H^i_\m(R(-\gamma)) \hookrightarrow H^i_\m(R^{1/p^e})$
%\cong H^i_\m(R) \oplus H^i_\m(M)
of degree zero. Let $v \in H^i_\m(R(-\gamma))_{a_i(R)}$ be an element in the top graded part of $H^i_\m(R(-\gamma))$, which has degree $a_i(R)+\gamma$. Under the inclusion above, this maps to a nonzero element of degree $a_i(R) + \gamma$ in $H^i_\m(R^{1/p^e})$. Therefore,
\[
\ds a_i(R) + \frac{b_\m(p^e)}{p^e} \ls a_i(R^{1/p^e}) = \frac{a_i(R)}{p^e},
\]
which is equivalent to
\[
\ds \frac{b_\m(p^e)}{p^e}\ls \frac{(1-p^e)a_i(R)}{p^e}.
\]

Since this holds for all   $e \gg 1$, we get
\[
\ds \fpt(R) = \lim_{e \to \infty} \frac{b_\m(p^e)}{p^e} 
\ls\lim_{e \to \infty}  -\frac{(p^e-1)a_i(R)}{p^e} =  -a_i(R)
\]
by Proposition \ref{PropEquivFPT}.
\end{proof}

\begin{corollary}\label{CorExtSW}
Let $(R,\m,K)$ be a standard graded $K$-algebra which is $F$-finite and $F$-pure. If $a_i(R)=0$ for some $i$, then $\sdim(R)=0$.
\end{corollary}
\begin{proof}
If $a_i(R)=0$ for some $i$, we have that $\fpt(R)=0$ by Theorem \ref{Thm a-inv}. Then, we have that 
$b_e=0$ for every $e\in \NN$ by Lemma \ref{Lemma Increase be} and Proposition \ref{PropEquivFPT}. As a consequence, $\m\subseteq I_e$ for every $e\in \NN.$ Since $I_e(R)\subseteq \m$ holds true because $R$ is $F$-pure, we have that $\m=I_e(R)$ for every $e\in \NN.$ Hence, $\cP(R)=\m,$ and $\sdim(R)=0.$ 
\end{proof}

We now review the definition of test ideal, which is closely related to the theory of tight closure. We refer the reader to \cite{HoHu2} for definitions and details.
\begin{definition}
Let $R$ be a Noetherian ring of positive characteristic $p$. The {\it finitistic test ideal of $R$} is defined as $\tau^{fg}(R) := \cap_M \ann(0^*_M)$, where $M$ runs through all the finitely generated $R$-modules. We define the {\it big test ideal of $R$} to be $\tau(R) = \cap_M \ann(0^*_M)$, where $M$ runs through all $R$-modules.
\end{definition}

\begin{remark}\label{Rem sp tau}
We point out that for $F$-finite rings, $\tau(R)$ is the smallest compatible ideal not contained in a minimal prime of $R$ \cite[Theorem 6.3]{KarlCentersFpurity}. In addition, $\tau^{fg}(R)$ is a compatible ideal  \cite[Theorem 3.1]{Vass}.
One clearly has the inclusion $\tau(R)\subseteq \tau^{fg}(R)$. 
It is one of the most important open problems in tight closure theory whether these two ideals are the same. 
Equality is known to hold true in some cases. For instance, Lyubeznik and Smith proved that $\tau^{fg}(R) = \tau(R)$ for finitely generated standard graded $K$-algebras \cite[Corollary 3.4]{FregEquiv}.
\end{remark}

We use Proposition \ref{PropEquivFPT} to relate the $F$-pure threshold of the ring with its splitting dimension.

\begin{theorem} \label{bounds_fpt_sdim}
Let $(R,\m,K)$ be a standard graded $K$-algebra which is $F$-finite and $F$-pure, and let $J\subseteq R$ be a compatible ideal. Then, we have 
$$
\fpt(R)\ls \fpt(R/J).
$$
In particular,
$$
\fpt(R)\ls\fpt(R/\tau)\hbox{  and }\fpt(R)\ls \fpt(R/\cP)\ls \sdim(R), 
$$
where $\tau$ denotes the test ideal of $R$, and $\cP$ the splitting prime of $R$.
\end{theorem}
\begin{proof}
Let $S=K[x_1,\ldots,x_n]$ be a polynomial ring such that there exists a surjection $S\to R$, and  let $\n=(x_1,\ldots,x_n)$, so that $\m = \n R$. Let $I$ denote the kernel of the surjection. Let $\widetilde{J}\subseteq S$ be the pullback of  $J$. 
We have that $(I^{[p^e]}:I)\subseteq (\widetilde{J}^{[p^e]}:\widetilde{J})$ for every $e\in\NN$ by Remark \ref{compatibleFedder}. Then,
$$
\min\left\{t\in\NN\ \bigg|\left[\frac{(\widetilde{J}^{[p^e]}:\widetilde{J})+\n^{[p^e]}}{\n^{[p^e]}}\right]_t\neq 0\right\}
\ls
\min\left\{t\in\NN\ \bigg| \left[\frac{(I^{[p^e]}:I)+\n^{[p^e]}}{\n^{[p^e]}}\right]_t\neq 0\right\}.
$$
As a consequence, we get
$$ 
b_\m(p^e)=\max\{ t\in\NN\mid \m^t\not\subseteq I_e(R)\} \ls \max\{ t\in\NN\mid \m^t\not\subseteq I_e(R/J)\}=b_{\m (R/J)}(p^e)
$$
by Proposition \ref{PropEquivFPT}. Then,
$\fpt(R)\ls\fpt(R/J)$.
The last claim follows from the fact that the test ideal is compatible as noted in Remark \ref{Rem sp tau}, and the splitting prime is compatible by Proposition \ref{splitting prime properties} (\ref{basic3}).  Finally,
 $\fpt(R/\cP) \ls \dim(R/\cP)$ \cite[Proposition 2.6(1)]{TW2004}. 
\end{proof}

We now focus on the diagonal $F$-threshold.
\begin{remark}\label{Lemma c(R)}
For any standard graded $K$-algebra $(R,\m,K)$, we have that
$$
\max\left\{s\in\frac{1}{p^e}\cdot\ZZ \ \bigg| \left[R^{1/p^e}/\m R^{1/p^e}\right]_s\neq 0\right\}=\frac{\nu_e}{p^e}.
$$
\end{remark}

We now are ready to proof the second part of Theorem \ref{MainTheorem}.
\begin{theorem}\label{Main c(R)}
Let $R$ be an $F$-finite standard graded $K$-algebra, and let $d=\dim(R)$. 
Then, $-a_d(R)\ls c^\m_{-}(R)$.
Furthermore, if $R$ is $F$-pure, then  $-a_i(R)\ls c^\m(R)$ for every $i$ such that $H^i_\m(R)\neq 0$.
\end{theorem}
\begin{proof}
We fix $i\in\NN$ such that $H^i_\m(R)\neq 0$. 
Let $v_1,\ldots,v_r$ be a minimal system of homogeneous generators of $R^{1/p^e}$ as an $R$-module, with degrees $\gamma_1,\ldots,\gamma_r\in\frac{1}{p^e}\NN$. By Remark \ref{frac_grading} we can view $R$ as a $\frac{1}{p^e} \NN$-graded module. We have a degree zero surjective map
$$
\bigoplus^r_{j=0} R(-\gamma_j)\stackrel{\phi}{\longrightarrow} R^{1/p^e},
$$
where $R(-\gamma_j)\to R^{1/p^e}$ maps $1$ to $v_j.$ 
This induces a degree zero homomorphism 
$$
\bigoplus^j_{i=0} H^i_\m(R(-\gamma_j))\stackrel{\varphi}{\longrightarrow} H^i_\m(R^{1/p^e}).
$$

If $i=d,$ $\varphi$ is surjective. We now prove that $\varphi$ is also surjective for $i\neq d$, if $R$ is $F$-pure. 
In this case, the natural inclusion $R\to R^{1/p^e}$ induces an inclusion $H^i_\m(R)\to H^i_\m( R^{1/p^e})$.
We have that the map $\theta:H^i_\m(R)\otimes_R R^{1/p^e}\to H^i_\m( R^{1/p^e})$
induced by $v\otimes f^{1/p^e}\mapsto f^{1/p^e} \alpha(v)$ is surjective \cite[Lemma 2.5]{AnuragUliPure}.
Then,
$$1\otimes \phi:H^i_\m(R)\otimes_R \left( \bigoplus^r_{j=0} R(-\gamma_j)\right)\to H^i_\m(R)\otimes_R R^{1/p^e}$$
is surjective. Thus, $\varphi$ is surjective, because $\varphi=\theta\circ (1\otimes \phi)$.

We have now that $\varphi$ is surjective under the hypotheses assumed.
Since $$\nu_\m(p^e)/p^e=\max\{\gamma_1,\ldots,\gamma_j\},$$ we have that
$$
\frac{a_i(R)}{p^e}=
a_i(R^{1/p^e})\ls
\max\{a_d(R(-\gamma_i)) \mid i=0,\ldots,j\}
=a_i(R)+\frac{\nu_\m(p^e)}{p^e}.
$$
Then, $a_i(R)\ls p^ea_d(R)+\nu_\m (p^e),$ and so
$(1-p^e)a_i(R)\ls \nu_\m(p^e)$.
Hence,
$$
-a_i(R)=\lim\limits_{e\to \infty}\frac{(1-p^e)a_i(R)}{p^e}\ls \liminf\limits_{e\to \infty}\frac{\nu_\m(p^e)}{p^e}=c^\m_{-}(R).
$$
Finally, we note that if $R$ is $F$-pure $c^\m_{-}(R)=c^\m(R).$
\end{proof}

%%%%%%%%%%%%%%%%%%%%%%%%%%%%%%%%%%%%%%%%%%%%%%%%%%%%%%%%%
\section{F-thresholds of graded Gorenstein rings}
%%%%%%%%%%%%%%%%%%%%%%%%%%%%%%%%%%%%%%%%%%%%%%%%%%%%%%%%%

Suppose that $(R,\m,K)$ is an $F$-finite standard graded Gorenstein $K$-algebra. Let $S=K[x_1,\ldots,x_n]$, and let $I \subseteq S$ be a homogeneous ideal such that $R\cong S/I$ as graded rings. Since $\Hom_R(R^{1/p},R)$ is a cyclic $R^{1/p}$-module,  we have that for all integers $e \gs 1$ there exist homogeneous polynomials $f_e \in S$ such that $I^{[p^e]}:I = f_eS + I^{[p^e]}$ by Remark \ref{CorrespFedder}. In fact, if $I^{[p]}:I = fS + I^{[p]}$, then $\ds I^{[p^e]}:I = f^{(1+p+ \cdots + p^{e-1})}S + I^{[p^e]}$ for all $e \gs 2$.

\begin{remark} \label{GorPrinc} When $R=S/I$ is $F$-pure, we have $(I^{[p^e]}:I) \not\subseteq \n^{[p^e]}$ by Fedder's criterion.  In the notation used above, if $(I^{[p^e]}:_SI)=f_eS+I^{[p^e]}$ for some homogeneous polynomial $f_e$, we get that
\[
\ds \min\left\{s\in\NN\ \bigg|\left[\frac{(I^{[p^e]}:I)+\n^{[p^e]}}{\n^{[p^e]}}\right]_s\neq 0\right\} = \min\left\{s\in\NN\ \bigg|\left[\frac{f_eS+\n^{[p^e]}}{\n^{[p^e]}}\right]_s\neq 0\right\} = \deg(f_e).
\]
\end{remark}

We now prove the last part of Theorem \ref{MainTheorem}.
\begin{theorem}\label{ThmGor}
Let $(R,\m,K)$ be a Gorenstein standard graded $K$-algebra which is $F$-finite and $F$-pure, and let $d = \dim (R)$. Then we have $\fpt(R)= -a_d(R)$.
\end{theorem} 
\begin{proof}
Let $S=K[x_1,\ldots,x_n]$ be a polynomial ring, and let $I \subseteq S$ be a homogeneous ideal such that $R \cong S/I$ as graded rings. Let $\n=(x_1,\ldots,x_n)$, so that $\m = \n R$.  
 Let $a=a_d(R)$.  Consider the natural map $S/I^{[p]} \to S/I$ induced by the inclusion $I^{[p]} \subseteq I$. Then such a map extends to a map of complexes $\psi_\bullet$ from a minimal free resolution of $S/I^{[p]}$ to a minimal free resolution of $S/I$. Furthermore, such a map $\psi_\bullet$ can be chosen graded of degree zero. We have that the last homomorphism in the map of complexes, $S(p(-n-a)) \to S(-n-a)$ is given by multiplication by a homogeneous polynomial $f$ (see Remark \ref{aInvBetti}). Furthermore,   $I^{[p]}:I = fS + I^{[p]}$  \cite[Lemma 1]{VraciuGorTightClosure}. Since $\psi_\bullet$ is homogeneous of degree zero, we have that $\deg(f) =  (p-1)(n+a)$.

Recall that, for all $e \gs 2$, we have that $(I^{[p^e]}:_S I) = f^{(1+p+\ldots +p^{e-1})}S + I^{[p^e]}$. By Remark \ref{GorPrinc} and Lemma \ref{Lemma Lift fpt}, we obtain that
\begin{align*}
\fpt(R)&=\lim\limits_{e\to\infty}\frac{n(p^e-1)-(\deg(f)\cdot (1+p+\ldots +p^{e-1}))}{p^e}\\
&=\lim\limits_{e\to\infty}\frac{n(p^e-1)}{p^e}-\lim\limits_{e\to\infty}\frac{\deg(f)\cdot (1+p+\ldots +p^{e-1})}{p^e}\\
&=n-\frac{\deg(f)}{p-1}\\
&=n-\frac{(p-1)(n+a)}{p-1}\\
&=-a.
\end{align*}
\end{proof}

We now give an example to show that an $F$-finite and $F$-pure standard graded $K$-algebra such that $\fpt(R) = -a_d(R)$ is not necessarily Gorenstein. This is not a counterexample to Conjecture \ref{conj} (2), since the ring we consider is not strongly $F$-regular. 

\begin{example} \label{exnotGor} Let $S=K[x,y,z]$ with $K$ a perfect field of characteristic $p>0$, and let $\n=(x,y,z)$ be its homogeneous maximal ideal. 
\[
\ds I=(xy,xz,yz) = (x,y) \cap (x,z) \cap (y,z) \subseteq S.
\]
Let $R=S/I$, with maximal ideal $\m=\n/I$. Note that $R$ is a one-dimensional Cohen-Macaulay $F$-pure ring. In addition, $\cP(R) = (x,y,z)R$; therefore, $\sdim(R) = 0$ and, by Theorem \ref{bounds_fpt_sdim}, $\fpt(R) = 0$ as well. On the other hand, from the short exact sequence 
\[
\xymatrixcolsep{5mm}
\xymatrixrowsep{2mm}
\xymatrix{
0 \ar[r] & \ds R \ar[rr] && \ds \frac{S}{(x,y)} \oplus \frac{S}{(x,z)\cap(y,z)} \ar[rr] && \ds \frac{S}{(x,y)+(x,z)\cap(y,z)} \cong K \ar[r]& 0
}
\]
we get a long exact sequence of local cohomology modules
\[
\xymatrixcolsep{5mm}
\xymatrixrowsep{2mm}
\xymatrix{
0 \ar[r] & K \ar[r] & H^1_\m(R) \ar[r] & H^1_\n(S/(x,y)) \oplus H^1_\n\left(S/(xy,z)\right) \ar[r] & \ldots
}
\]
The maps in this sequence are homogeneous of degree zero. Thus, $a_1(R) \gs 0$, because $K$ injects into $H^1_\m(R)$. On the other hand, since $R$ is $F$-pure, we have that $a_1(R) \ls 0$; therefore, $\fpt(R) = a_1(R) =0$. Although $R$ is not Gorenstein, since the canonical module $\omega_R \cong (x,y)/(xy,xz+yz)$ has two generators.
\end{example}

Sannai and Watanabe \cite[Theorem 4.2]{SWJPAA} showed that for an $F$-pure standard graded  Gorenstein algebra, $R$, with an isolated singularity, $\sdim(R)=0$ is equivalent to $a_d(R)=0.$ 
The previous theorem recovers this result dropping the hypothesis of isolated singularity. This is because
$\sdim(R)=0$ is equivalent to $\fpt(R)=0.$ In fact, for all $F$-pure rings, Corollary \ref{CorExtSW} shows that $a_d(R)=0$ implies
$\sdim(R)=0.$

We now aim at an interpretation of the $F$-pure threshold of a standard graded Gorenstein $K$-algebra as the maximal length of a regular sequence that preserves $F$-purity.

\begin{proposition}\label{PropFpureSeq}
Let $S=K[x_1,\ldots, x_n]$ be a polynomial ring over an $F$-finite infinite field $K$. Let $\n=(x_1,\ldots,x_n)$ denote the maximal homogeneous ideal.
Let $I\subseteq S$ be a homogeneous ideal such that  $R=S/I$ is an $F$-pure ring, and let $\m=\n R$. Let $f\in (I^{[p]}:I)\smallsetminus \n^{[p]}$. If $\deg(f)\ls (p-1)(n-1)$, then there exists a linear form $\ell\in S$  such that:
\begin{enumerate}
\item  $\ell^{p-1}f\not\in \n^{[p]}$. 
\item the class of $\ell$ in $R$ does not belong to $\cP(R)$.
\item $\ell$ is  a non-zero divisor in $R$.
\end{enumerate}
\end{proposition}
\begin{proof}
Let us pick $c_\alpha\in K$ such that $f=\sum_{|\alpha |=\deg(f)} c_\alpha x^{\alpha}$. Let 
$$
\ell_y=y_1x_1+\ldots +y_n x_n\in S[y_1,\ldots,y_n]
$$
be a generic linear form. We note that 
$$
\left(\ell_y\right)^{p-1}=\sum_{|\theta|=p-1} g_\theta(y)x^{\theta},
$$ 
where $g_\theta(y)=\frac{(p-1)!}{\theta_1 !\cdots \theta_n!}y^\theta \in K[y_1,\ldots,y_n]$.

Since $f\not\in \n^{[p]},$ there exists $x^\beta\in\Supp \{f\}$ such that $x^\beta\not \in \n^{[p]}$. Since $|\beta| \leq (p-1)(n-1)$ and $x^{\beta}\not\in \n^{[p]}$, there exists $x^{\gamma}\in \n^{p-1}$ such that $x^{\gamma}x^{\beta}\not\in \n^{[p]}$ by Pigeonhole Principle. Let 
$$
h:=\sum_{\beta+\gamma=\theta+\alpha} c_{\alpha}g_\theta (y)\in K[y_1,\ldots,y_n]
$$
We note that $h\neq 0$ because $c_\beta g_\theta\neq 0$. In addition, $h$ is the coefficient of $x^{\theta+\gamma}$ in $\left(\ell_y\right)^{p-1} f$. We note that $\cP(R)\neq \m$ by Lemma \ref{LemmaSP}, and thus $\cP(R)\cap \m \neq \m$. Since $K$ is an infinite field, we can pick a point $v\in K^n$ such that $h(v)\neq 0$ and the class of $\ell_y(v)$ does not belong to $\cP(R)$. We set $\ell=\ell_y (v).$ By our construction of $\ell$, $x^{\beta+\gamma}\in\Supp\{\ell^{p-1}f\}$ and  $x^{\beta+\gamma}\not\in \n^{[p]}$. In addition, $\ell \notin\cP(R).$ Since the pullback of $\cP(R)$ to $S$ contains every associated prime of $R$, we have that $\ell$ is a non-zero divisor in $R$.
\end{proof}

Note that, for $\ell$ as in Proposition \ref{PropFpureSeq}, if we set $I':=I+(\ell)$ we have that the ring $S/I'$ is again $F$-pure. In fact, for $f$ as above, we have that $\ell^{p-1}f \in (I'^{[p]}:I') \smallsetminus \n^{[p]}$, and $F$-purity follows by Fedder's criterion \cite[Theorem 1.12]{Fedder}.

As a consequence of these results, and of Theorem \ref{ThmGor}, we give an interpretation of the $F$-pure threshold, and the $a$-invariant, in terms of the maximal length of a regular sequence that preserves $F$-purity. We start by introducing the concept of $F$-pure regular sequence.

\begin{definition}\label{Def F-pure seq}
Let $R$ be an $F$-finite $F$-pure ring. We say that a regular sequence $f_1,\ldots, f_r$ is $F$-pure if 
$R/(f_1,\ldots,f_i)$ is an $F$-pure ring for all  $i=1,\ldots,r.$ 
\end{definition}

\begin{lemma}\label{LemmaInqRegSep}
Let $(R,\m,K)$ be a standard graded $K$-algebra. If $f$ is a regular element of degree $d>0$, then $d+a_{i}(R) \ls a_{i-1}(R/(f))$ for all $i \in \NN$ such that $H^i_\m(R) \ne 0$.
\end{lemma}
\begin{proof}
Suppose that $H^i_\m(R) \ne 0$. Consider the homogeneous short exact sequence
\[
\xymatrixcolsep{5mm}
\xymatrixrowsep{2mm}
\xymatrix{
0 \ar[rr] && R(-d) \ar[rr]^-{ f} && R \ar[rr]&& R/(f) \ar[rr] && 0.
}
\]
For all $j \in \ZZ$, this gives rise to an exact sequence of $K$-vector spaces
\[
\xymatrixcolsep{5mm}
\xymatrixrowsep{2mm}
\xymatrix{
\ldots \ar[rr] && H^{i-1}_\m(R/(f))_j \ar[rr] && H^i_\m(R)_{j-d} \ar[rr] && H^i_\m(R)_j\ar[rr] && \ldots
}
\]
Since $d>0$, for $j = a_i(R) + d$ we have that $H^i_\m(R)_j = 0$. Then,  
\[
\xymatrixcolsep{5mm}
\xymatrixrowsep{2mm}
\xymatrix{
H^{i-1}_\m(R/(f))_{a_i(R)+d} \ar[rr]&& H^i_\m(R)_{a_i(R)} \ar[rr]&& 0
}
\]
is a surjection. We note that  $H^i_\m(R)_{a_i(R)} \ne 0$, which yields $H^{i-1}_\m(R/(f))_{a_i(R) + d} \ne 0$, and hence $a_{i-1}(R/(f)) \gs a_i(R)+d$. 
\end{proof}
\begin{corollary} \label{CorInqRegSep} Let $(R,\m,K)$ be a standard graded $K$-algebra which is $F$-finite and $F$-pure. If $f_1,\ldots,f_r$ is a homogeneous $F$-pure regular sequence of degrees $d_1,\ldots,d_r$, then $\sum_{j=1}^r d_j \ls \min\{- a_i(R) \mid i \in \NN\}$.
\end{corollary}
\begin{proof}
We proceed by induction on $r \gs 1$. Assume that $r=1$. If $H^i_\m(R) = 0$, we have that $d_1 \ls -a_i(R) = \infty$, therefore there is nothing to prove in this case. If $H^i_\m(R) \ne 0$, by Lemma \ref{LemmaInqRegSep} we have that $a_i(R)+d_1 \ls a_{i-1}(R/(f_1))$. Since $R/(f_1)$ is $F$-pure, it follows from Remark \ref{FpureNeg} that $a_{i-1}(R/(f)) \ls 0$, and hence $d_1 \ls -a_i(R)$. Thus, $d_1 \ls -a_i(R)$ for all $i \in \NN$, that is, $d_1 \ls \min\{ -a_i(R) \mid i \in \NN\}$. This concludes the proof of the base case. For $r > 1$, if $H^i_\m(R) = 0$ we have that $\sum_{i=1}^r d_i \ls -a_i(R) = \infty$ and, again, there is nothing to prove in this case. Assume that $H^i_\m(R) \ne 0$. By induction, we get that $\sum_{j=2}^r d_j \ls -a_s(R/(f_1))$ for all $s \in \NN$. In particular, we have that $\sum_{j=2}^r d_j \ls -a_{i-1}(R/(f_1))$. By Lemma \ref{LemmaInqRegSep}, we have that $-a_{i-1}(R/(f_1)) \gs -a_i(R) - d_1$. Combining the two inequalities, and rearranging the terms in the sum, we obtain $\sum_{j=1}^r d_i \ls -a_i(R)$. Therefore, we obtain $\sum_{j=1}^r d_j \ls \min\{-a_i(R) \mid i \in \NN\}$.
\end{proof}

\begin{theorem}\label{ThmFpureSeq}
Let $(R,\m,K)$ be a Gorenstein standard graded $K$-algebra which is $F$-finite and $F$-pure, and let $d=\dim(R)$. 
If $f_1,\ldots,f_r$ is an $F$-pure regular sequence, then $r\ls \fpt(R)$. Furthermore, if $K$ is infinite, then there exists an  $F$-pure regular sequence consisting of $\fpt(R)$  linear forms.
\end{theorem}
\begin{proof}
By Theorem \ref{ThmGor}, we have that $\fpt(R) = -a_d(R)$. The first claim follows from Corollary \ref{CorInqRegSep}. For the second claim, let $S=K[x_1,\ldots,x_n]$ be a polynomial ring and let $I \subseteq S$ be a homogeneous ideal such that $R \cong S/I$ as graded rings. 
We proceed by induction on $\fpt(R)$. The case $\fpt(R) = 0$ is trivial. We now assume $\fpt(R)> 0$. From the proof of Theorem \ref{ThmGor}, we have that $(I^{[p]}:_S I) = fS+I^{[p]}$ for a homogeneous polynomial $f\in (I^{[p]}:_SI) \smallsetminus \n^{[p]}$ of degree $\deg(f) \ls (p-1)(n+a_d(R))$. Since $a_d(R)=-\fpt(R)<0$ by assumption,  there exists a linear non-zero divisor $\ell_1 \in R$ such that $R/(\ell_1)$ is $F$-pure by Proposition \ref{PropFpureSeq}. Note that, from the homogeneous short exact sequence
\[
\xymatrixcolsep{5mm}
\xymatrixrowsep{2mm}
\xymatrix{
0 \ar[r] & H^{d-1}_\m(R/(\ell_1)) \ar[r] & H^d_\m(R)(-1) \ar[r]^-{\ell_1} & H^d_\m(R) \ar[r] & 0,
}
\]
it follows that $a_{d-1}(R/(\ell_1)) = a_d(R)+1$. Since $R/(\ell_1)$ is  Gorenstein, we have that $\fpt(R/(\ell_1)) = -a_{d-1}(R/(\ell_1)) = -a_d(R) - 1 = \fpt(R)-1$. The claim follows by induction.
\end{proof}

\section{Results in characteristic zero} \label{char0}

In this section we present results in characteristic zero that are analogous to Theorem \ref{Thm a-inv}. These results are motivated by the relation between the log-canonical and the $F$-pure threshold.

 We first fix the notation. Let $K$ be a field of characteristic zero, and let $(R,\m,K)$ be a $\QQ$-Gorenstein normal standard graded $K$-algebra. Consider the closed subscheme $V(\m) = Y \subseteq X = \Spec(R)$, and let $\a$ be the corresponding ideal sheaf. 
We now use Hironaka's resolution of singularities \cite{Hironaka}.
Suppose that $f:\widetilde X \to X$ is a log-resolution of the pair $(X,Y)$, that is, $f$ is a proper birational morphism with $\widetilde X$ nonsingular such that the ideal sheaf $\a\mathcal{O}_{\widetilde{X}} = \OO_{\widetilde{X}}(-F)$ is invertible,  and $\Supp(F) \cup {\rm Exc}(f)$ is a simple normal crossing divisor.  Let $K_X$ and $K_{\widetilde{X}}$ denote canonical divisors of X and $\widetilde X$, respectively.

Let $\lambda \gs 0 $ be a real number. Then there are finitely many irreducible (not necessarily exceptional) divisors $E_i$ on $\widetilde X$ and real numbers $a_i$ so that there exists an $\RR$-linear equivalence of $\RR$-divisors
\[
\ds K_{\widetilde{X}} \sim f^* K_X + \sum_i a_i E_i + \lambda F.
\]
\begin{definition}
Continuing with the previous notation, we say that the pair $(X, \lambda Y)$ is log canonical, or lc for short, if $a_i \gs -1$ for all $i$. Define
\[
\ds \lct(X) = \sup \{\lambda  \in \RR_{\gs 0}   \mid \mbox{ the pair } (X,\lambda Y) \mbox{ is lc}\}.
\]
We say that  $(X, \lambda Y)$ is Kawamata log-terminal, or  klt for short, if $a_i>-1 $ for
all $i.$
\end{definition}
\begin{remark}\label{Rem lct-klt}
If $X$ is log-terminal, we have that
\[
\ds \lct(X) = \sup \{\lambda  \in \RR_{\gs 0}   \mid \mbox{ the pair } (X,\lambda Y) \mbox{ is klt}\}.
\]
\end{remark}

\begin{definition}
Let $K$ be a field of positive characteristic $p$, and let $(R,\m,K)$ be a standard graded $K$-algebra which is $F$-finite and $F$-pure. Let $I \subseteq R$ be a homogeneous ideal. For a real number $\lambda\gs 0$, we say that $(R,I^\lambda)$ is \emph{ strongly $F$-regular} if for every $c \in R$ not in any minimal prime, there exists $e \gs 0$ and an element $f\in I^{\lceil p^e\lambda\rceil}$ such that the inclusion of $R$-modules $(cf)^{1/p^e}R\subseteq R^{1/p^e}$ splits.
\end{definition} 

\begin{remark}\label{Rem sfr-fp}
Let $(R,\m,K)$ be a standard graded $K$-algebra which is $F$-finite and strongly $F$-regular.  
Then, by \cite[Proposition 2.2 (5)]{TW2004}, we have
$$\fpt(R)=\sup \{\lambda  \in \RR_{\gs 0}   \mid \mbox{the pair } (R,\m^{\lambda}) \mbox{ is strongly }F\mbox{-regular}\}.$$ 
\end{remark}

\begin{definition}
Let $R$ be a reduced algebra essentially of finite type over a field $K$ of characteristic zero, $\a \subseteq R$ an ideal, and $\lambda > 0$ a real number.
The pair $(R,\a^\lambda)$ is said to be of dense $F$-pure type (respectively strongly $F$-regular type) if there exist a finitely generated $\ZZ$-subalgebra $A$ of $K$ and a reduced subalgebra $R_A$ of $R$ essentially of finite type over $A$ which satisfy the following conditions:
\begin{enumerate}[(i)]
\item $R_A$ is flat over $A$, $R_A \otimes_A K \cong R$, and $\a_AR = \a$, where $\a_A = \a \cap R_A \subseteq R_A$.
\item The pair $(R_s, \a_s^\lambda)$ is $F$-pure (respectively strongly $F$-regular) for every closed point $s$ in a dense subset of $\Spec(A)$. Here, if $\kappa(s)$ denotes the residue field of $s$, we define $R_s = R_A \otimes_A \kappa(s)$ and $\a_s = \a_AR_s \subseteq R_s$.
\end{enumerate}
\end{definition}

\begin{remark} \label{aInvBetti} Let $S = K[x_1,\ldots,x_n]$ be a polynomial ring over a field $K$, and $I \subseteq S$ be a homogeneous ideal. Let $R= S/I$, and consider a minimal graded free resolution of $R$ over $S$
\[
\xymatrixcolsep{5mm}
\xymatrixrowsep{2mm}
\xymatrix{
0 \ar[r] & G_{n-d} \ar[r] & G_{n-d-1} \ar[r] & \cdots\cdots \ar[r] & S \ar[r] & R \ar[r] & 0.
}
\]
If we write $G_{n-d} = \bigoplus_j S(-j)^{\beta_{n-d,j}(R)}$, where the positive integers $\beta_{n-d,j}(R)$ are the $(n-d)$-th graded Betti numbers of $R$ over $S$, we have that $\max\{j \mid \beta_{n-d,j}(R) \ne 0\} = -n-a_d(R)$ \cite[Section 3.6]{BrHe}. Therefore, when $R$ is Cohen-Macaulay, $a_d(R)$ can be read from the graded Betti numbers of $R$ over $S$.
\end{remark}
\begin{lemma}\label{Lemma Red a-in}
Let $(R,\m,K)$ be a Cohen-Macaulay standard graded $K$-algebra of dimension $d$, with $K$ a field of characteristic zero, and let $A$ be a finitely generated $\ZZ$-subalgebra of $K$. Assume that $R_0$ is a graded $A$-algebra such that $R_0 \otimes_A K \cong R$ as graded rings. For a closed point $s \in \Max\Spec(A)$, let $R_s=R_0 \otimes_A \kappa(s)$, where $\kappa(s)$ is the residue field of $s \in \Spec(R)$. Then, there exists a dense open subset $U \subseteq \Max\Spec(A)$ such that $a_d(R) = a_d(R_s)$ for all $s \in U$. 
\end{lemma} 
\begin{proof}
Since $R_0$ is a finitely generated graded $A$-algebra, we  write $R_0 \cong B/J$, for some homogeneous ideal $J \subseteq B := A[x_1,\ldots,x_n]$. Let $T = B \otimes_A K \cong K[x_1,\ldots,x_n]$, and for $s \in \Max\Spec(R)$, let $B_s = B \otimes_A \kappa(s) \cong \kappa(s)[x_1,\ldots,x_n]$. By \cite[Theorem 2.3.5, and Theorem 2.3.15]{HHCharZero}, there exists a dense open subset $U \subseteq \Max\Spec(R)$ such that, for all $s \in U$, $R_s$ is Cohen-Macaulay of dimension $d = \dim(R)$. Furthermore, the graded Betti numbers of $R$ over $T$ are the same as the graded Betti numbers of $R_s$ over $B_s$ \cite[Theorem 2.3.5 (e)]{HHCharZero}. In particular, it follows from Remark \ref{aInvBetti} that $a_d(R) = a_d(R_s)$ for all $s \in U$.
\end{proof}

\begin{theorem}\label{MainCharZero}
Let $K$ be a field of characteristic zero, and let $(R,\m,K)$ be a standard graded normal and $\QQ$-Gorenstein $K$-algebra such that $X=\Spec R$ is log-terminal. Let $d=\dim(R)$. Then,
\begin{enumerate}
\item $\lct(X)\ls -a_d(R)$. 
\item If $R$ is Gorenstein, then $\lct(X)=-a_d(R).$
\end{enumerate}
\end{theorem}
\begin{proof}
We can write $R=K[x_1,\ldots,x_n]/(f_1,\ldots,f_\ell)$ for some integer $n$ and some homogeneous polynomials $f_1,\ldots,f_\ell \in S:=K[x_1,\ldots,x_n]$.  Let $A$ be the finitely generated $\ZZ$-algebra generated by all the coefficients of $f_1,\ldots, f_\ell$. Define $T:=A[x_1,\ldots,x_n]$ and notice that, if we set $R_0:=T/(f_1,\ldots,f_\ell)$, we have that $R_0 \otimes_A K \cong R$. Since $X$ is log-terminal, $R$ is a Cohen-Macaulay ring.
For a closed point $s \in \Spec(A)$, set $R_{s}:=R_0\otimes_A \kappa(s)$, and $\m_s: =(x_1,\ldots,x_n)R_{s}$.
For $m \in \NN$, let $\lambda_m=\lct(X)-\frac{1}{m}$.
Then, the pair $(X,\lambda_m Y)$ is klt by Remark \ref{Rem lct-klt}.
Thus,  $(X,\lambda_m Y)$ is of dense strongly $F$-regular type \cite[Theorem 6.8]{H-Y} \cite[Corollary 3.5]{TakagiAdj}.
It follows that $(R_{s},\m_s^{\lambda_m})$ is strongly $F$-regular for each closed point $s\in V$, where $V \subseteq \Max\Spec(A)$ is a dense set. By Remark \ref{Rem sfr-fp}, we have that $\lambda_m \ls \fpt(R_{s})$  for all $m$, and hence
\[
\ds \lambda_m\ls \fpt(R_{s}) \ls -a_d(R_{s})
\]
by Theorem \ref{Thm a-inv}.
By Lemma \ref{Lemma Red a-in}, there exists a dense open subset $U \subseteq \Max\Spec(A)$ such that $a_d(R_s) = a_d(R)$ for all $s \in U$. Thus, for $s$ in non-empty intersection $U \cap V$, we have
\[
\ds \lct(X)-\frac{1}{m} \ls -a_d(R_{s}) = -a_d(R).
\]
After taking the limit as $m\to \infty$, we obtain that $\lct(X)\ls -a_d(R)$.

For the second result, we note that if $R$ is a Gorenstein ring, 
then $R_{s}$ is a Gorenstein ring for $s$ in a dense open subset $W \subseteq \Max\Spec(A)$ \cite[Theorem 2.3.15]{HHCharZero}. Let $U$ be as above.
Then, $\fpt(R_{s})=-a_d(R_{s})=-a_d(R)$ for $s \in W \cap U$, where $W\cap U$ is a dense open subset of $\Max\Spec(A)$. Let $\delta_m:=-a_d(R)-\frac{1}{m}$. Because of the equality obtained above, we have that $(R,\m^{\delta_m}_{s})$ is $F$-pure for $s \in W$. Thus, $\delta_m\ls\lct(X)$ for every $m \in \NN$ \cite[Proposition 3.2 (1)]{TW2004}, and hence $-a_d(R)\ls\lct(X)$. The desired equality now follows from the first part.
\end{proof}

A key point in the previous theorem was that if $X$ is log-terminal, then
$$\lct(X)=\sup \{\lambda  \in \RR_{\gs 0}   \mid \mbox{ the pair } (X,\lambda Y) \mbox{ is klt}\},$$ and a pair is klt if and only if the pair is of dense strongly $F$-regular type.  One could try to replace log-terminal by log-canonical in the statement of Theorem \ref{MainCharZero}, to get an analogue of Theorem \ref{MainTheorem} in characteristic zero. It is a very important and longstanding open problem whether a pair is log-canonical if and only if the pair is of dense $F$-pure type (see \cite{TakagiAdj,HaraWatanabe,MS}). In addition, note that Lemma \ref{Lemma Red a-in} requires that $X$ is Cohen-Macaulay, and this is not necessarily implied by $X$ being log-canonical. It is not  clear to us that the $a$-invariant behaves well with respect to reduction to positive characteristic, in case $X$ is only log-canonical. Motivated by these considerations, we make the following conjecture:

\begin{conjecture}
Let $(R,\m,K)$ be a standard graded normal  $\QQ$-Gorenstein algebra over a field, $K$, of characteristic zero. Let $d=\dim(R)$ 
and let $X=\Spec R$. 
Suppose that $Y$ is log-canonical. 
Then,
\begin{enumerate}
\item $\lct(X)\ls -a_d(R)$.
\item If $R$ is Gorenstein, then $\lct(X)=-a_d(R).$
\end{enumerate}
\end{conjecture}

We also ask the following question motivated by Theorem \ref{ThmFpureSeq}.

\begin{question}
Let $(R,\m,K)$ be a standard graded normal  Gorenstein algebra over a field, $K$, of characteristic zero. Let $d=\dim(R)$, $a=a_d(R)$, and let $X=\Spec R$. Suppose that $X$ is log-canonical. 
Then, are there $a$ linear forms, $\ell_1,\ldots, \ell_a,$ such that
$$
\Spec R/(\ell_1,\ldots, \ell_t)
$$
is log-canonical?
\end{question}

%%%%%%%%%%%%%%%%%%%%%%%%%%%%%%%%%%%%%%%%%%%%%%%%%%%%%%%%%%%%%%%%%%%%%%%%%%%%
\section{Homological invariants of $F$-pure rings}
%%%%%%%%%%%%%%%%%%%%%%%%%%%%%%%%%%%%%%%%%%%%%%%%%%%%%%%%%%%%%%%%%%%%%%%%%%%%

\label{SecProjDim}
Let $K$ be a field and let $S=K[x_1,\ldots, x_n]$ be a polynomial ring. Let $I\subseteq S$ be a homogeneous ideal, and let $R=S/I$.
Suppose that $I=(f_1,\ldots,f_j)$ is generated by forms of degree $d_i = \deg(f_i)$.  Let $G_\bullet$ be the
minimal graded free resolution of $R$. Each $G_i$ can be written as a direct sum of copies of $S$ with shifts:
$$
G_i=\bigoplus_{j\in\ZZ} S(-j)^{\beta_{i,j}(R)},
$$
where $S(-j)$ denotes a rank one free module where the generator has degree
$j$.  The exponents $\beta_{i,j}(R)$ are called the graded Betti numbers of $R$.
We define the projective dimension of $R$ by
\[
\ds \pd_S(R) = \max\{i \mid \beta_{i,j}(R) \neq 0 \text{ for some $j$}\}.\]  
The Castelnuovo-Mumford regularity of $R=S/I$ is defined by
$$
\ds \reg_S(R) = \max\{j-i \mid \beta_{i,j}(R) \neq 0
\hbox{ for some $i$}\}.
$$
Equivalently, if $d=\dim(R)$, it can be defined as
\[
\ds \reg_S(R) = \max\{a_i(R) + i \mid i=0,\ldots,d\}.
\]
Suppose that $R$ is an $F$-pure ring.
In this section, we provide bounds for the projective dimension and Castelnuovo-Mumford regularity of $R$ over $S$. This relates to an important question in commutative algebra asked by Stillman: 

\begin{question}[{\cite[Problem 3.14]{PeevaStillman}}]\label{QStillman} 
Let $S=K[x_1,\ldots, x_n]$ be a polynomial ring over a field $K$ and fix a sequence of natural numbers $d_1,\ldots,d_j$.  Does there exist a constant $C=C(d_1,\ldots,d_j)$ (independent of $n$) such that
\[\pd_S(S/J) \ls C\]
for all homogeneous ideals $J \subseteq S$ generated by homogeneous polynomials of degrees $d_1,\ldots,d_j$?
\end{question}

Recall that $b_\m(p^e) = \max\{r \ \mid \m^r \not\subseteq I_e(R)\}$. Proposition \ref{PropFpureSeq} gives a relation between $b_\m(p)$ and $\Depth_S(R)$, hence between $b_\m(p)$ and $\pd_S(R)$ by the Auslander-Buchsbaum's formula. For $F$-pure rings,  the projective dimension and the Castelnuovo-Mumford regularity have explicit upper bounds. The bound for the projective dimension easily follows from a special case of a result of Lyubeznik.  This argument has already been used, essentially, in \cite[Theorem 4.1]{AnuragUliPure}. 
\begin{theorem} \cite[Corollary 3.2]{LyuVan} \label{lyubFdepth} Let $S$ be regular local ring of positive characteristic $p$, and let $I \subseteq S$ be an ideal. Let $R=S/I$, with maximal ideal $\m$. Then, for $i \in \NN$, we have $H^{n-i}_I(S) = 0$ if and only if $F^e:H^i_\m(R) \to H^i_\m(R)$ is the zero map for some $e \in \NN$.
\end{theorem}

We now present a theorem that gives bounds for homological invariants of  standard graded algebras $F$-pure $K$-algebras.

\begin{theorem} \label{bounds}
Let $S=K[x_1,\ldots, x_n]$ be a polynomial ring over a field of positive characteristic $p$. Let $I \subseteq S$ be a homogeneous ideal such that $R=S/I$ is $F$-pure. Then, 
$$\pd_S(R)\ls \mu_S(I) ,$$
where $\mu_S(I)$ denotes the minimal number of generators of $I$ in $S$.
If $K$ is $F$-finite, then $$\reg_S(R)\ls \dim(R)-\fpt(R).$$
\end{theorem}

\begin{proof}
Since $R$ is an $F$-pure ring, so is the localization $(R',\m'):=(R_\m,\m_\m)$. In fact if $R \subseteq R^{1/p}$ is pure, then so is $R_\m \subseteq (R^{1/p})_\m = (R_\m)^{1/p}$. Hence the Frobenius homomorphism acts injectively on the local cohomology modules of $R'$. In particular, for any integer $i$ and for any $e \in \NN$, we have that $F^e: H^i_{\m'}(R') \to H^i_{\m'}(R')$ is the zero map if and only if $H^i_{\m'}(R') = 0$. Let $\n=(x_1,\ldots,x_n)$. Since $I$ is homogeneous, we have that $H^{n-i}_I(S) = 0$ if and only if $H^{n-i}_{I_\n}(S_\n) = 0$. By Theorem \ref{lyubFdepth}, it follows that $H^i_{\m'}(R') = 0$ for all $n-i > \cd(I,S)$, and $H_{\m'}^{n-\cd(I,S)}(R') \ne 0$. Therefore, by the Auslander-Buchsbaum's formula, we get $\pd_S(R)=n-\Depth_{S_\n}(R') = \cd(I,S)$. Since $\cd(I,S) \ls \mu_S(I)$, the first claim follows.
For the second claim, let $d=\dim(R)$. Then we have that 
\begin{align*}
\reg_S(R)&=\max\{a_i(R)+i \mid i=0,\ldots,d\}\\
&\ls \max\{i-\fpt(R)\mid i=0,\ldots,d\}\\
&=d-\fpt(R),
\end{align*}
where the second line follows from Theorem \ref{Thm a-inv}.
\end{proof}
\begin{remark} In the notation introduced above, Caviglia proved that finding an upper bound $C=C(d_1,\ldots,d_j)$ for $\pd_S(R)$ is equivalent to finding an upper bound $B=B(d_1,\ldots,d_j)$ for $\reg_S(R)$ (see \cite[Theorem 29.5]{PeevaGrSyz} and \cite[Theorem 2.4]{McCSec}). The bound for $\reg_S(R)$ that we give in Theorem \ref{bounds}, a priori, depends on $n$. However, the inequality $\pd_S(R) \ls \mu_S(I)$ for the projective dimension shows that Stillman's question has positive answer for $F$-pure rings. In particular, there also exists $B=B(d_1,\ldots,d_j)$ such that $\reg_S(R) \ls B$.
\end{remark}

Motivated by the results in the previous theorem, we ask the following question.

\begin{question}
Let $S=K[x_1,\ldots,x_n]$ be a polynomial ring over a field, $K$, of characteristic zero, and let $I$ be a homogeneous ideal.
Suppose that $R=S/I$ is a normal and $\QQ$-Gorenstein ring, and that $X=\Spec R$ is log-canonical.
Is it true that 
$$
\pd_S(R)\ls \mu_S(I)\hbox{ and }\reg_S(R)\ls \dim(R)-\lct(X)\hbox{?}
$$
\end{question}

We end this section by exhibiting an $S_k$ condition that forces an $F$-pure ring to be Cohen-Macaulay.
\begin{proposition}\label{PropSerre}
Let $S=K[x_1,\ldots,x_n]$ be a polynomial ring over a field of positive characteristic $p$. Let $I$ be a homogeneous ideal generated by forms $f_1,\ldots, f_\ell.$ Let $D=\deg(f_1)+\ldots+\deg(f_\ell).$ Suppose that $R=S/I$ is $F$-pure. If $R_Q$ is Cohen-Macaulay for every prime ideal such $\dim(R_Q)\ls D$ , then $R$ is Cohen-Macaulay. 
\end{proposition}
\begin{proof}
Our proof will be by contradiction. 
Suppose that $R$ is not Cohen-Macaulay.
Then, $\depth(R)<\dim(R).$
Let $c=\cd(I,S)$ and $r=\depth_I(S).$
We have that
$$r=n-\dim(R)<n-\depth(R)=\pd_S(R)=c$$
by Theorem \ref{bounds}. 
Let $Q\in\Ass_SH^c_I(S)$. Note that $H^c_I(S_Q)\neq 0$ by our choice of $Q$, and then $c=\cd(IS_Q, S_Q)$. In addition, $H^r_I(S_Q)\neq 0$ because  $I\subseteq Q$. Thus, $r=\depth_I(S_Q)$. It follows that
$$r=n\dim(S_Q)-\dim(R_Q)<\dim(S_Q)-\depth(R)=\pd_{S_Q}(R_Q)=c$$
by Theorem \ref{bounds}. Thus, $\dim(R_Q)\neq \depth(R_Q)$, and so $R_Q$ is not Cohen-Macaulay.

We have that 
$\dim S/Q\gs n-D$ \cite[Theorem 1]{YiZhang}, therefore $\dim S_{Q}\ls D$,
because regular rings satisfy the dimension formula. In particular, $\dim(R_Q)\ls D$. Since we are assuming that $R$ is $S_D$, we have that $R_Q$ must be Cohen-Macaulay, and we reach a contradiction.
 \end{proof}

\section*{Acknowledgments} 
We thank Craig Huneke for very helpful discussions. In particular, we thank him for suggesting a shorter and more elegant proof of Theorem \ref{ThmGor}. We thank Karl Schwede for discussions about geometric interpretations of Theorem \ref{ThmFpureSeq}. We also thank Andrew Bydlon for reading a preliminary version of this article, and for useful feedback and comments.

\bibliographystyle{alpha}
\bibliography{References}

\begin{thebibliography}{HMTW08}

\bibitem[AE05]{AE}
Ian~M. Aberbach and Florian Enescu.
\newblock The structure of {F}-pure rings.
\newblock {\em Math. Z.}, 250(4):791--806, 2005.

\bibitem[Amb99]{Ladder}
F.~Ambro.
\newblock Ladders on {F}ano varieties.
\newblock {\em J. Math. Sci. (New York)}, 94(1):1126--1135, 1999.
\newblock Algebraic geometry, 9.

\bibitem[BH93]{BrHe}
Winfried Bruns and J{\"u}rgen Herzog.
\newblock {\em Cohen-{M}acaulay rings}, volume~39 of {\em Cambridge Studies in
  Advanced Mathematics}.
\newblock Cambridge University Press, Cambridge, 1993.

\bibitem[CM12]{fptHibi}
Takahiro Chiba and Kazunori Matsuda.
\newblock Diagonal {F}-thresholds and {F}-pure thresholds of hibi rings.
\newblock {\em arXiv:arXiv:1201.5691}, 2012.

\bibitem[Fed83]{Fedder}
Richard Fedder.
\newblock {$F$}-purity and rational singularity.
\newblock {\em Trans. Amer. Math. Soc.}, 278(2):461--480, 1983.

\bibitem[GW78]{GW1}
Shiro Goto and Keiichi Watanabe.
\newblock On graded rings. {I}.
\newblock {\em J. Math. Soc. Japan}, 30(2):179--213, 1978.

\bibitem[Her12]{DanielMRL}
Daniel~J. Hern{\'a}ndez.
\newblock {$F$}-purity of hypersurfaces.
\newblock {\em Math. Res. Lett.}, 19(2):389--401, 2012.

\bibitem[HH94]{HoHu2}
Melvin Hochster and Craig Huneke.
\newblock {$F$}-regularity, test elements, and smooth base change.
\newblock {\em Trans. Amer. Math. Soc.}, 346(1):1--62, 1994.

\bibitem[HH99]{HHCharZero}
Melvin Hochster and Craig Huneke.
\newblock Tight closure in equal characteristic zero.
\newblock 1999.

\bibitem[Hir64]{Hironaka}
Heisuke Hironaka.
\newblock Resolution of singularities of an algebraic variety over a field of
  characteristic zero. {I}, {II}.
\newblock {\em Ann. of Math. (2) 79 (1964), 109--203; ibid. (2)}, 79:205--326,
  1964.

\bibitem[HMTW08]{HMTW}
Craig Huneke, Mircea Musta{\c{t}}{\u{a}}, Shunsuke Takagi, and Kei-ichi
  Watanabe.
\newblock F-thresholds, tight closure, integral closure, and multiplicity
  bounds.
\newblock {\em Michigan Math. J.}, 57:463--483, 2008.
\newblock Special volume in honor of Melvin Hochster.

\bibitem[HR76]{HRFpurity}
Melvin Hochster and Joel~L. Roberts.
\newblock The purity of the {F}robenius and local cohomology.
\newblock {\em Advances in Math.}, 21(2):117--172, 1976.

\bibitem[HW02]{HaraWatanabe}
Nobuo Hara and Kei-Ichi Watanabe.
\newblock F-regular and {F}-pure rings vs. log terminal and log canonical
  singularities.
\newblock {\em J. Algebraic Geom.}, 11(2):363--392, 2002.

\bibitem[HWY14]{HWY}
Daisuke Hirose, Kei-ichi Watanabe, and Ken-ichi Yoshida.
\newblock {$F$}-thresholds versus {$a$}-invariants for standard graded toric
  rings.
\newblock {\em Comm. Algebra}, 42(6):2704--2720, 2014.

\bibitem[HY03]{H-Y}
Nobuo Hara and Ken-Ichi Yoshida.
\newblock A generalization of tight closure and multiplier ideals.
\newblock {\em Trans. Amer. Math. Soc.}, 355(8):3143--3174 (electronic), 2003.

\bibitem[Li13]{LiSocles}
Jinjia Li.
\newblock Asymptotic behavior of the socle of {F}robenius powers.
\newblock {\em Illinois J. Math.}, 57(2):603--627, 2013.

\bibitem[LS99]{FregEquiv}
Gennady Lyubeznik and Karen~E. Smith.
\newblock Strong and weak {$F$}-regularity are equivalent for graded rings.
\newblock {\em Amer. J. Math.}, 121(6):1279--1290, 1999.

\bibitem[Lyu06]{LyuVan}
Gennady Lyubeznik.
\newblock On the vanishing of local cohomology in characteristic {$p>0$}.
\newblock {\em Compos. Math.}, 142(1):207--221, 2006.

\bibitem[Ma13]{MaRFmod}
Linquan Ma.
\newblock Finiteness properties of local cohomology for {$F$}-pure local rings.
\newblock {\em International Mathematics Research Notices}, 2013.

\bibitem[MS11]{MS}
Mircea Musta{\c{t}}{\u{a}} and Vasudevan Srinivas.
\newblock Ordinary varieties and the comparison between multiplier ideals and
  test ideals.
\newblock {\em Nagoya Math. J.}, 204:125--157, 2011.

\bibitem[MS13]{McCSec}
Jason McCullough and Alexandra Seceleanu.
\newblock Bounding projective dimension.
\newblock In {\em Commutative algebra}, pages 551--576. Springer, New York,
  2013.

\bibitem[NBS14]{NBS}
Luis N{\'u}{\~n}ez-Betancourt and Ilya Smirnov.
\newblock {H}ilbert-{K}unz multiplicities and {$F$}-thresholds.
\newblock {\em Preprint}, 2014.

\bibitem[Pee11]{PeevaGrSyz}
Irena Peeva.
\newblock {\em Graded syzygies}, volume~14 of {\em Algebra and Applications}.
\newblock Springer-Verlag London, Ltd., London, 2011.

\bibitem[PS09]{PeevaStillman}
Irena Peeva and Mike Stillman.
\newblock Open problems on syzygies and {H}ilbert functions.
\newblock {\em J. Commut. Algebra}, 1(1):159--195, 2009.

\bibitem[Sch10]{KarlCentersFpurity}
Karl Schwede.
\newblock Centers of {$F$}-purity.
\newblock {\em Math. Z.}, 265(3):687--714, 2010.

\bibitem[SW07]{AnuragUliPure}
Anurag~K. Singh and Uli Walther.
\newblock Local cohomology and pure morphisms.
\newblock {\em Illinois J. Math.}, 51(1):287--298 (electronic), 2007.

\bibitem[SW11]{SWJPAA}
Akiyoshi Sannai and Kei-ichi Watanabe.
\newblock {$F$}-signature of graded {G}orenstein rings.
\newblock {\em J. Pure Appl. Algebra}, 215(9):2190--2195, 2011.

\bibitem[SZ13]{BertiniFsing}
Karl Schwede and Wenliang Zhang.
\newblock Bertini theorems for {$F$}-singularities.
\newblock {\em Proc. Lond. Math. Soc. (3)}, 107(4):851--874, 2013.

\bibitem[Tak04]{TakagiAdj}
Shunsuke Takagi.
\newblock F-singularities of pairs and inversion of adjunction of arbitrary
  codimension.
\newblock {\em Invent. Math.}, 157(1):123--146, 2004.

\bibitem[TW04]{TW2004}
Shunsuke Takagi and Kei-ichi Watanabe.
\newblock On {F}-pure thresholds.
\newblock {\em J. Algebra}, 282(1):278--297, 2004.

\bibitem[Vas98]{Vass}
Janet~Cowden Vassilev.
\newblock Test ideals in quotients of {$F$}-finite regular local rings.
\newblock {\em Trans. Amer. Math. Soc.}, 350(10):4041--4051, 1998.

\bibitem[Vra03]{VraciuGorTightClosure}
Adela Vraciu.
\newblock Tight closure and linkage classes in {G}orenstein rings.
\newblock {\em Math. Z.}, 244(4):873--885, 2003.

\bibitem[Zha11]{YiZhang}
Yi~Zhang.
\newblock A property of local cohomology modules of polynomial rings.
\newblock {\em Proc. Amer. Math. Soc.}, 139(1):125--128, 2011.

\end{thebibliography}

{\footnotesize

\noindent \small \textsc{Department of Mathematics, University of Virginia, Charlottesville, VA  22903} \\ \indent \emph{Email address}:  {\tt ad9fa@virginia.edu} 
\vspace{.25cm}

\noindent \small \textsc{Department of Mathematics, University of Virginia, Charlottesville, VA  22903} \\ \indent \emph{Email address}:  {\tt lcn8m@virginia.edu}  

}

\end{document}